\documentclass[11pt]{article}
\usepackage{amsmath, amssymb, amsfonts, amstext, amsthm, textcomp, enumerate, xcolor,comment,mathtools,authblk}
\usepackage[mathscr]{euscript}
\usepackage[colorlinks=true]{hyperref}
\newtheorem{thm}{Theorem}[section]
\newtheorem{lem}[thm]{Lemma}
\newtheorem{cor}[thm]{Corollary}
\newtheorem{pro}[thm]{Proposition}
\newtheorem{defn}[thm]{Definition}
\newtheorem{rem}[thm]{Remark}
\newtheorem{ex}[thm]{Example}

\usepackage[top=3cm,left=3cm,right=3cm,bottom=5.5cm]{geometry}
\usepackage{witharrows}
\usepackage{color}

\DeclareMathOperator{\hash}{^{\#}}

\DeclareMathOperator{\tr}{\text{trace}}
\DeclareMathOperator{\Ind}{\emph{Ind }}
\newcommand\blfootnote[1]{%
  \begingroup
  \renewcommand\thefootnote{}\footnote{#1}%
  \addtocounter{footnote}{-1}%
  \endgroup
}
\providecommand{\keywords}[1]
{  \small	
  \textbf{\textit{Keywords---}} #1
}

\title{Generalized Inverse Preservers of Hadamard Circulant Majorization}
\author[1]{C. C. Hsu}
\affil[1]{No. 605, Daxue S. Rd., Nanzi District, Kaohsiung City, Taiwan}
\author[2]{P. R. Raickwade}
\author[2]{K. C. Sivakumar\thanks{Corresponding author}}
\affil[2]{Department of Mathematics, Indian Institute of Technology Madras, Chennai, 600036, Tamil Nadu, India.}

 \date{}

\begin{document}
\maketitle
\blfootnote{Email addresses:
mcatch7269@gmail.com (C. C. Hsu),
 pavanraickwade@gmail.com (P. R. Raickwade), kcskumar@iitm.ac.in (K. C. Sivakumar). }
\hrule
\abstract{ Let $T\colon M_n\rightarrow M_n$ preserve Hadamard circulant majorization. In this note, we show that this property is inherited by the three most popular generalized inverses, viz. the Moore-Penrose inverse, the group inverse and the Drazin inverse.}

\vspace{0.1in}
\hrule
\vspace{0.1in}
\section{Introduction and preliminaries}\label{sec1}

Let $M_n$ be the space of all real square matrices of order $n$ and $\mathbb{N}_n := \{1, \ldots, n\}$. For $i, j \in \mathbb{N}_n$, let $E_{ij}$ be the matrix in $M_n$ whose $(i, j)^\text{th}$ entry is $1$ and $0$ otherwise. A matrix $D \in M_n$ with nonnegative entries is called {\it doubly stochastic} if each of its rows and columns sums to 1. Let
\[ C := \left[\begin{array}{rrrr}
{\bf e_n}, {\bf e_1}, & \cdots, {\bf e_{n-1}}
\end{array} \right], \]
where ${\bf e_i}$ is the $i^\text{th}$ standard basis vector in $\mathbb{R}^n$ i.e. the $i^\text{th}$ coordinate of ${\bf e_i}$ is equal to $1$, while all its other coordinates are $0$. For $i \in \mathbb{N}_n$, we define $C_i = C^i$. Let $\mathscr{C}_n:=\{C_1, \ldots, C_n\}$. A matrix $C \in M_n$ is called {\it circulant doubly stochastic} if it is a convex combination of $C_1, \ldots, C_n$. Circulant doubly stochastic matrices are, evidently, a special case of doubly stochastic matrices. For $j \in \mathbb{N}_n$, let $\sigma_{j}$ be the circulant permutation associated naturally with the circulant permutation matrix $C_j$. More precisely, $\sigma_1 = (\begin{matrix}
1 & 2 & \cdots & n
\end{matrix})$ is an element of the symmetric group $S_n$, and $\sigma_j = \sigma_1^j$ for each $j \in \mathbb{N}_n$. So, for each $j \in \mathbb{N}_n$, $C_j$ can also be written as
\[ C_j = \left[ \begin{array}{c}
({\bf e_{\sigma_j(1)}})^\top \\
\vdots \\
({\bf e_{\sigma_j(n)}})^\top
\end{array} \right]. \]
For $X = [x_{ij}]$, $Y = [y_{ij}] \in M_n$, the {\it Hadamard product} of $X$ and $Y$ is defined by $X \odot Y := [x_{ij} y_{ij}]$.    
For $X, Y \in M_n$, we say that $X$ is {\it Hadamard majorized} by $Y$, denoted by $X\prec_H Y$, if there exists a doubly stochastic matrix $D\in M_n$ such that $X=D\odot Y$. Next, a stronger notion is recalled. $X$ is said to be {\it Hadamard circulant majorized} by $Y$, denoted by $X \prec_{HC} Y$, if there exists a circulant doubly stochastic matrix $C \in M_n$ such that $X = C \odot Y$. Let $T: M_n \rightarrow M_n$ be a linear operator. Then $T$ is said to preserve {\it Hadamard majorization} if 
\begin{center}
$T(X)\prec_H T(Y)$, whenever $X\prec_H Y$.     
\end{center}
Similalry, we say that $T$ preserves {\it Hadamard circulant majorization} if 
\begin{center}
$T(X) \prec_{HC} T(Y)$, whenever $X \prec_{HC} Y$.
\end{center}

The concept of Hadamard circulant majorization was introduced and investigated in \cite{Had_circ_maj}. It is also interesting to note that this notion does not generalize the usual concept of majorization of vectors. For more details, refer \cite{Had_circ_maj,Hadamard_majorization}, where for instance, the authors {\it also} discuss the case of stronger requirements than those that are given here.

It is easy to see that $X\prec_{HC} Y$ implies $X\prec_H Y$. However, it must be clarified that the notions of linear preserver of Hadamard circulant majorization and that of (just) Hadamard majorization are {\it not comparable}. We show this in Example \ref{counter1} and Example \ref{counter2}. We recall two results, which will be useful in our discussion.  The first
one presents a necessary condition for Hadamard majorization, the second one provides a necessary condition for an operator to preserve Hadamard circulant majorization, while the third result gives a necessary and sufficient condition for an operator to preserve Hadamard circulant majorization.

\begin{thm}[{\cite[Theorem 3.11]{Hadamard_majorization}}]\label{Had_maj}
    Let $n\geq 3$ and $T\colon M_n\rightarrow M_n$ a linear operator. If $T$ preserves Hadamard majorization then   $T(E_{ij})\odot T(E_{kl})=0$ for every $1\leq i,j,k,l\leq n$ with $(i,j)\neq (k,l).$
     
\end{thm}

\begin{lem}[{\cite[Lemma 3]{Had_circ_maj}}]\label{Cj_dominated_C_Pj}
    Let $T\colon M_n\rightarrow  M_n$ be a linear operator that preserves Hadamard circulant majorization. Then there exists a permutation $P$ on $\mathbb{N}_n$ such that, for each $j\in \mathbb{N}_n$, $T(E_{i\sigma_j(i)})$ is dominated by $C_{P(j)}$ for every $i\in \mathbb{N}_n.$
\end{lem}

\begin{thm}[{\cite[Theorem 4]{Had_circ_maj}}] \label{Had_circ_maj_charcaterization)}
Let $T: M_n \rightarrow M_n$ be  linear. Then the following statements are equivalent:
\begin{enumerate}
\item $T$ preserves Hadamard circulant majorization.
\item Let $C_k \in \mathscr{C}_n, B \in M_n$ and $B_k = B \odot C_k$. Then, for some circulant doubly stochastic matrix $C_{B_k}$ in $M_n$, we have $$T(B_k) = C_{B_k} \odot T(B).$$ 
\end{enumerate}
\end{thm}

\begin{ex}\label{counter1}
    Consider $T\colon M_3\rightarrow M_3$ defined as \[T(A):=\begin{bmatrix}
        a_{11}&0&0\\
        0&0&a_{11}\\
        0&a_{11}&0
    \end{bmatrix},\quad A=[a_{ij}]\in M_3.\] For a doubly stochastic matrix $D=[d_{ij}]\in M_3$, define a doubly stochastic matrix $\hat{D}\in M_3$ as $\hat{D}=\begin{bsmallmatrix}
        d_{11}&d_{12}&d_{13}\\
        d_{12}&d_{13}&d_{11}\\
        d_{13}&d_{11}&d_{12}
    \end{bsmallmatrix}$. Then, for every $Y\in M_3$, \[T(D\odot Y)=\hat{D}\odot T(Y).\] Hence, if $X\prec_H Y$, then for some doubly stochastic $D$, we would have $X=D \odot Y.$ It then follows that $T(X)=\hat{D}\odot T(Y),$ for the doubly stochastic matrix $\hat{D}$, constructed as above. Thus, $T$ preserves Hadamard majorization. Next, note that $T(C_1\odot I)=0$, but 
    \[ C_1\odot T(I)=E_{23},~ C_2\odot T(I)=E_{32} \quad \text{and} \quad C_3\odot T(I)=E_{11}.\] Hence, by Theorem \ref{Had_circ_maj_charcaterization)}, we conclude that $T$ does not preserve Hadamard circulant majorization.
\end{ex}

\begin{ex}\label{counter2}
    Let $T\colon M_3 \rightarrow M_3$ be defined by
\[ T(A) = \left[ \begin{array}{ccc}
a_{12} + a_{23} + a_{31} & a_{11} & 0 \\
0 & a_{12} - a_{23} + a_{31} & a_{22} \\
a_{33} & 0 & a_{31} \\
\end{array} \right],\quad A\in M_3. \]
Observe that
\[T(C_1\odot A)=C_3\odot T(A),~T(C_2\odot A)=C_2\odot T(A),~T(C_3\odot A)=C_1\odot T(A).\]
Hence $T$ preserves Hadamard cirulant majorization, by Theorem \ref{Had_circ_maj_charcaterization)}. Further, the entry along the intersection of the  first row and the first column of the matrix $T(E_{12})\odot T(E_{23})$ is $1$ and so it is not the zero matrix. Hence, by Theorem \ref{Had_maj}, it follows that $T$ does not preserve Hadamard majorization.
\end{ex}

In order to motivate the contents of this note, in what follows, we recall the three notions of generalized inverses. For $X, Y \in M_n$, define $\left\langle X, Y \right\rangle := \tr(XY^\top)$, where $Y^\top$ means the transpose of $Y$. Then, $M_n$, along with this inner product, can be regarded as a Hilbert space. Let $T\colon M_n\rightarrow M_n$ be a linear operator. Denote the adjoint operator of $T$ by $T^*$ and, for a subset $S\subseteq M_n$, let $S^{\perp}$ denote the orthogonal complement of $S$. It is well known that there is a unique operator $X\colon M_n\rightarrow M_n$ satisfying the following conditions: $$TXT=T,\; XTX=X,\; (TX)^*=TX,\; (XT)^*=XT.$$ This unique operator $X$ is called the {\it Moore-Penrose inverse} of $T$ and it is denoted by $T^\dag$.
There are several equivalent definitions of the Moore-Penrose inverse, see \cite{Ben_israel_book,Generalized_inverses_singapore}. For example, we shall make use of the following characterization: 

\begin{pro}\label{MP_equivalence_defn}
Let $T\colon M_n \rightarrow M_n$ be linear. Then $T^{\dagger}\colon M_n\rightarrow M_n$ is the unique linear operator satisfying:
\begin{enumerate}
\item $T^\dagger T(X) = X$ for every $X \in N(T)^{\perp}$, 
\item $T^\dagger (Y) = 0$ for every $Y \in R(T)^{\perp}$.
\end{enumerate}
\end{pro}


\begin{defn}
    For a linear operator $T\colon M_n\rightarrow M_n$, the smallest nonnegative integer $k$ for which $\emph{rank } T^k=\emph{rank }T^{k+1}$ is called the index of $T$. It is denoted by $\emph{Ind } T$.
\end{defn}
A linear operator on a finite dimensional vector space has a finite index. If $\text{Ind } T= m$, then there exists a unique linear operator $U\colon M_n\rightarrow M_n$ such that \begin{align*}
    T^mUT=T^m,~UTU=U,~ UT=TU.
\end{align*} 
This unique operator $U$ is called the {\it Drazin inverse} of $T$ and it is denoted by $T^D$. Moreover, we have, $M_n = N(T^m) \oplus R(T^m)$ and $T|_{R(T^m)}$ is an injective mapping from $R(T^m)$ onto $R(T^m)$. When $m=1$, the Drazin inverse has a special name; the {\it group inverse}. The group inverse of $T$ is denoted by $T\hash$.  Let us emphasize that unlike the Moore-Penrose inverse or the Drazin inverse, the group inverse need not exist for a given operator. For instance, no nonzero nilpotent operator has the group inverse. For more details, refer to \cite{Ben_israel_book,Generalized_inverses_singapore,gen_inv_line_transformation},

The following equivalent definition of the Drazin inverse will be useful in our context.

\begin{pro} [{\cite[Corollory 12.1.2]{Generalized_inverses_singapore}}]\label{equivalence_Drazin_inverse}
Let $T\colon M_n \rightarrow M_n$ be a linear operator with $\Ind T = m$. Then the Drazin inverse of $T$ is the unique linear operator $T^D$ satisfying:
\begin{enumerate}
\item $T^D (X) = 0$ for every $X \in N(T^m)$,
\item $T^DT (Y) = TT^D (Y) = Y$ for every $Y \in R(T^m)$.
\end{enumerate}
\end{pro}

Let us turn our attention to the main result of \cite{Had_circ_maj}, which is the motivation for the study undertaken here.\\

\begin{thm}[{\cite[Theorem 6]{Had_circ_maj}}]\label{motiv}
Let $T: M_n \rightarrow M_n$ be a linear bijection. If $T$ preserves Hadamard circulant majorization, then so does $T^{-1}$.    
\end{thm}

Next, we present a summary of our results. In view of Theorem \ref{motiv}, we ask if there are analogues which may be presented in terms of the three generalized inverses stated above, when the operator $T$ is not bijective. Our results show that the answer is in the affirmative. We obtain verbatim versions of Theorem \ref{motiv} for three of the most prominent generalized inverses; the Drazin inverse (Theorem \ref{Drazin}), the group inverse (Corollary \ref{gr}), and the Moore-Penrose inverse (Theorem \ref{MP}). We show that the adjoint operator inherits the preserver property, too (Theorem \ref{adj}). It is pertinent to point to the rather surprising fact, that very little research has been undertaken towards determining if generalized inverses preserve a given majorization property. In this context, we refer to \cite{kcsmarsusa} for perhaps the first and recent results on this topic. As is mentioned earlier, the notion of Hadamard circulant majorization has no resemblance to the usual notion of majorization and so the results obtained here may not be considered as generalizations of those reported in \cite{kcsmarsusa}.

\section{Main results}
We begin with a set of auxiliary results.\\

\begin{lem}\label{Lemma 1}
Let $A \in M_n$ and $i \in \mathbb{N}_n$. If $A \odot C_i = A$, then $A \odot C_j = 0$ for all $j \in \mathbb{N}_n \setminus \{i\}$.
\end{lem}
\begin{proof}
Set $A:=[a_{hk}].$
Then $A = A \odot C_i =  \sum_{h = 1}^n a_{h \sigma_i(h)} E_{h \sigma_i(h)}.$ Thus, $a_{hk} = 0$ for all $k \neq \sigma_i(h)$. Let $j \in \mathbb{N}_n\setminus\{i\}$. Then
\[
(A \odot C_j)_{hk} = a_{hk}(C_j)_{hk} =
\begin{cases} 
    a_{h \sigma_j(h)}, & \text{if } k = \sigma_j(h) \\
    0, & \text{otherwise}
\end{cases}.
\]
Also, we have $\sigma_j(h) \neq \sigma_i(h)$. Hence, we obtain that $a_{h \sigma_j(h)} = 0$. Therefore $A \odot C_j = 0$, as required.
\end{proof}
\begin{lem}\label{unknown}
    Let $T\colon M_n\rightarrow M_n$ be a linear preserver of Hadamard circulant majorization. Then there exists a permutation $P$ on $\mathbb{N}_n$ such that, for every $C_j\in \mathscr{C}_n$ and $B\in M_n$,
    \begin{align*}
        T(C_j\odot B)= C_{P(j)}\odot T(B).
    \end{align*}
\end{lem}
\begin{proof}
    For $B = [b_{hk}]$, we have $B = \displaystyle \sum_{h, k} b_{hk} E_{hk}$. Let $P$ be the permutation on $\mathbb{N}_n$ as described in Lemma \ref{Cj_dominated_C_Pj}. Then
\begin{align*}
C_{P(j)} \odot T(B) &= C_{P(j)} \odot \displaystyle \left( \sum_{h, k} b_{hk} T(E_{hk}) \right) \\
& = \displaystyle \sum_{h, k} b_{hk} \left( C_{P(j)} \odot T(E_{hk}) \right) \\
& = \displaystyle \sum_{h = 1}^n \left( b_{h \sigma_j(h)} \left( C_{P(j)} \odot T(E_{h \sigma_j(h)}) \right) + \displaystyle \sum_{k \neq \sigma_j(h)}  b_{hk} \left( C_{P(j)} \odot T(E_{hk}) \right) \right).
\end{align*}
Note that, $k \neq \sigma_j(h)$ implies $k = \sigma_r(h)$ for some $r \in \mathbb{N}_n \setminus \{j\}$. Therefore, $T(E_{hk})$ is dominated by $C_{P(r)}$, and thus, from Lemma \ref{Lemma 1}, we get $C_{P(j)} \odot T(E_{hk}) = 0.$ Since $$C_{P(j)} \odot T(E_{h \sigma_j(h)}) = T(E_{h \sigma_j(h)}),$$ we get
\begin{align*}
C_{P(j)} \odot T(B) & = \displaystyle \sum_{h = 1}^n \left( b_{h \sigma_j(h)} \left( C_{P(j)} \odot T(E_{h \sigma_j(h)}) \right) + \displaystyle \sum_{k \neq \sigma_j(h)}  b_{hk} \left( C_{P(j)} \odot T(E_{hk}) \right) \right) \\
& = \displaystyle \sum_{h = 1}^n \left( b_{h \sigma_j(h)} T(E_{h \sigma_j(h)}) \right) \\
& = T \left( \displaystyle \sum_{h = 1}^n b_{h \sigma_j(h)} E_{h \sigma_j(h)} \right) \\
& = T(C_j \odot B).
\end{align*}
\end{proof}

\begin{lem}\label{invariance_of N(T)_and_R(T)}
Let $T\colon M_n \rightarrow M_n$ be a linear preserver of Hadamard circulant majorization. Let $K$ denote one of the four subspaces: $N(T), R(T), N(T)^{\perp}, R(T)^{\perp}.$ Then, for any circulant doubly stochastic matrix $C$, we have the following implication:\begin{equation*}
    X \in K \Longrightarrow C\odot X\in K.
\end{equation*}

\end{lem}
\begin{proof}
For $X\in N(T)$, set $Z:=C\odot X$. Then $Z\prec_{HC} X$ and hence, $T(Z)\prec_{HC} T(X)$. Thus, there exists a circulant doubly stochastic matrix $C'$ such that $T(Z)=T(C\odot X)=C'\odot T(X)=0$. Thus, $C\odot X\in N(T).$

Next, we consider the subspace $R(T)$. Let $X = T(B)$ for some $B \in M_n$ and $C =  \sum_{i = 1}^n r_i C_i$ for some nonnegative real numbers $r_1, \ldots, r_n$ with $ \sum_{i = 1}^n r_i = 1$. Let $P$ be a permutation on $\mathbb{N}_n$ as described in Lemma \ref{unknown}. Thus, for each $j\in \mathbb{N}_n$, $T(C_{P^{-1}(j)} \odot B) = C_{j} \odot T(B)$.  Hence 
    \[C \odot X = \left( \sum_{i = 1}^n r_i C_i \right) \odot T(B) = \sum_{i = 1}^n r_i \left( C_i \odot T(B) \right) = \sum_{i = 1}^n r_i T(C_{P^{-1}(i)} \odot B) \in R(T). \] 
    
We now prove the implication for $K:=N(T)^{\perp}.$ The proof for $K:=R(T)^{\perp}$ is similar. We have, for $X = [x_{ij}], Y = [y_{ij}], Z = [z_{ij}] \in M_n$, 
\begin{equation}\label{<C.X,Y=X,C.Y>}
    \left\langle X \odot Y, Z \right\rangle = \displaystyle \sum_{i = 1}^n \sum_{j = 1}^n x_{ij} y_{ij} z_{ij} = \left\langle  X, Y \odot Z \right\rangle.
\end{equation} Now, let $X\in N(T)^\perp$. It follows, from the arguments given above, that if $Z\in N(T),$ then $C\odot Z\in N(T).$ Hence, $\left\langle C \odot X, Z \right\rangle = \left\langle X, C \odot Z \right\rangle = 0$. Therefore $C \odot X \in N(T)^{\perp},$ as required.
\end{proof}

The assumption of $C$ being a circulant doubly stochastic matrix, in the result above, is necessary; see the next example.

\begin{ex}
    Let $T\colon M_3 \rightarrow M_3$ be defined by
\[ T(X) = \begin{bmatrix}
x_{12} & x_{11} & 0 \\
0 & x_{12} & x_{22} \\
x_{33} & 0 & x_{31}  \\
\end{bmatrix},\quad X=[x_{ij}]\in M_3.  \]
It is easily verified that $$T(C_1\odot X)=C_3\odot T(X),~T(C_2\odot X)=C_2\odot T(X),~T(C_3\odot X)=C_1\odot T(X).$$ Hence, from Theorem \ref{Had_circ_maj_charcaterization)}, it follows that $T$ preserves Hadamard circulant majorization. Next, set 
\begin{center}
$X:=\begin{bsmallmatrix}
    1&0&0\\
    0&1&0\\
    0&0&0
\end{bsmallmatrix}$ and $P:=\begin{bsmallmatrix}
    1&0&0\\
    0&0&1\\
    0&1&0
\end{bsmallmatrix}$.    
\end{center}
Then $X\in R(T)$. However, $P\odot X=\begin{bsmallmatrix}
    1&0&0\\
    0&0&0\\
    0&0&0
\end{bsmallmatrix}\notin R(T).$
\end{ex}

We are now ready to prove the results mentioned in the introduction. First, we show that the adjoint is a preserver, whenever the operator is a preserver.\\

\begin{thm}\label{adj}
Let $T\colon M_n \rightarrow M_n$ be linear. If $T$ preserves Hadamard circulant majorization, then so does $T^*$.
\end{thm}
\begin{proof}
Let $C_k \in \mathscr{C}_n$ and $B \in M_n$. We write $B = X + Y$ for $X \in R(T)$, $Y \in R(T)^{\perp}$. By Lemma \ref{invariance_of N(T)_and_R(T)}, we have $C_k \odot Y \in R(T)^{\perp} = N(T^*)$ and so, $T^*(C_k \odot B) = T^*(C_k \odot X)$. Next, let $Z \in M_n$ be decomposed as $Z = Z_1 + Z_2$, with $Z_1 \in N(T)$ and $ Z_2 \in N(T)^{\perp}$. Let $P$ be a permutation on $\mathbb{N}_n$ as described in Lemma \ref{unknown}. Then,
\begin{DispWithArrows*}
    \left\langle T^*(C_k \odot X), Z \right\rangle &= \left\langle C_k \odot X, T(Z) \right\rangle \Arrow{From Eq. \eqref{<C.X,Y=X,C.Y>}}\\
    &=\left\langle X, C_k \odot T(Z) \right\rangle \\
&= \left\langle X, C_k \odot T(Z_2) \right\rangle \Arrow{From Lemma \ref{unknown}} \\
&= \left\langle X, T(C_{P^{-1}(k)} \odot Z_2) \right\rangle \\
&=\left\langle T^*(X), C_{P^{-1}(k)} \odot Z_2 \right\rangle \\
&= \left\langle C_{P^{-1}(k)} \odot T^*(X), Z_2 \right\rangle.
\end{DispWithArrows*}
Since $T^*(X) \in R(T^*) = N(T)^{\perp}$, it follows, from Lemma \ref{invariance_of N(T)_and_R(T)}, that $C_{P^{-1}(k)} \odot T^*(X) \in N(T)^{\perp}$. Therefore, $\left\langle C_{P^{-1}(k)} \odot T^*(X), Z_1 \right\rangle = 0$ and thus $$\left\langle T^*(C_k \odot X), Z \right\rangle = \left\langle C_{P^{-1}(k)} \odot T^*(X), Z \right\rangle.$$ Since $Z \in M_n$ is arbitrary, we have
\[T^*(C_k \odot B)=T^*(C_k \odot X) = C_{P^{-1}(k)} \odot T^*(X) = C_{P^{-1}(k)} \odot T^*(B). \]
By Theorem \ref{Had_circ_maj_charcaterization)}, we may now conclude that $T^*$ preserves Hadamard circulant majorization.
\end{proof}

Next, we consider the Drazin inverse.

\begin{thm}\label{Drazin}
Let $T\colon M_n \rightarrow M_n$ be a linear operator preserving Hadamard circulant majorization. Then, $T^D$ inherits that property.
\end{thm}
\begin{proof}
Let $\text{Ind } T = m$. Let $C_k \in \mathscr{C}_n$ and $B \in M_n$, with $B = X + Y$ for some $X \in N(T^m)$ and $Y \in R(T^m)$. It is easy to see that, if $T$ preserves Hadamard circulant majorization, then so does $T^m.$ Hence, from Lemma \ref{invariance_of N(T)_and_R(T)} and Proposition \ref{equivalence_Drazin_inverse},
\begin{align*}
T^D(C_k \odot B) &= T^D(C_k \odot X) + T^D(C_k \odot Y) \\ & = T^D(C_k \odot Y)\\ &= (T|_{R(T^m)})^{-1}(C_k \odot Y). 
\end{align*}
Since $Y \in R(T^m) = R(T^{m+1})$, it follows, from $M_n = N(T^m) \oplus R(T^m)$, that $Y = T^{m+1}(Z)$ for some $Z \in R(T^m)$. Let $P$ be the permutation on $\mathbb{N}_n$ as described in Lemma \ref{unknown}. Then
\begin{align*}
(T|_{R(T^m)})^{-1}(C_k \odot T^{m+1}(Z))&= (T|_{R(T^m)})^{-1}(T^{m+1}(C_{(P^{m+1})^{-1}(k)} \odot Z)) \\
&= T^m (C_{(P^{m+1})^{-1}(k)} \odot Z) \\
&= C_{P^{-1}(k)} \odot T^m(Z).
\end{align*}
Also, we have 
\begin{align*}
T^D(B) & = T^D(X) + T^D(Y)\\ & = T^D(Y) \\ & = T^DT(T^m(Z)) \\& = T^m(Z).     
\end{align*}
Hence
\begin{align*}
T^D(C_k \odot B) = C_{P^{-1}(k)} \odot T^D(B).
\end{align*}
From Theorem \ref{Had_circ_maj_charcaterization)}, it then follows  that $T^D$ preserves Hadamard circulant majorization.
\end{proof}

The following consequence for the group inverse is immediate. \\

\begin{cor}\label{gr}
    Let $T\colon M_n\rightarrow M_n$ be a group invertible linear operator. If $T$ preserves the Hadamard circulant majorization, then so does $T\hash$.
\end{cor}
Here is an example to illustrate Corollary \ref{gr}.
\begin{ex}
Let $T: M_3 \rightarrow M_3$ be defined by
\[ T(X) = \begin{bmatrix}
x_{12} + x_{23} + x_{31} & x_{11} & 0 \\
0 & x_{12} - x_{23} + x_{31} & x_{22} \\
x_{33} & 0 & x_{31} \\
\end{bmatrix},\quad X=[x_{ij}]\in M_3. \]
It is  easily verified that $T\hash$ exists and is given by
\[ T^{\#}(X) = \begin{bmatrix}
x_{12} & \frac{(x_{11} + x_{22})}{2} - x_{33} & 0 \\
0 & x_{23} & \frac{(x_{11} - x_{22})}{2} \\
x_{33} & 0 & x_{31} \\
\end{bmatrix}, \quad X\in M_3. \] Further, observe that if $S:=T$ or $S:=T^{\#}$, then $S$ satisfies 
$$S(C_1\odot X)=C_3\odot S(X),~S(C_2\odot X)=C_2\odot S(X),~S(C_3\odot X)=C_1\odot S(X).$$
Therefore, both $T$ and $T^{\#}$ preserve Hadamard circulant majorization, by Theorem \ref{Had_circ_maj_charcaterization)}.
\end{ex}
\vspace{0.3cm}

We conclude this note with the result for the Moore-Penrose inverse.\\

\begin{thm}\label{MP}
Let $T: M_n \rightarrow M_n$ be linear. If $T$ preserves Hadamard circulant majorization, then so does $T^{\dagger}$.
\end{thm}
\begin{proof}
Let $C_k \in \mathscr{C}_n$ and $B \in M_n$ with $B = X + Y$ for some $X \in R(T)$, $Y \in R(T)^{\perp}$. By Lemma \ref{invariance_of N(T)_and_R(T)} applied to $K:=R(T)^\perp$, $C_k \odot Y \in R(T)^{\perp}$, and so, $T^{\dagger}(C_k \odot B) = T^{\dagger}(C_k \odot X)$. Since $X \in R(T)$, it follows, from $M_n = N(T) \oplus N(T)^{\perp}$, that $X = T(A)$ for some $A \in N(T)^{\perp}$. Let $P$ be a permutation on $\mathbb{N}_n$ as described in Lemma \ref{unknown}. Then \begin{DispWithArrows*}
    T^{\dagger}(C_k \odot X) &= T^{\dagger}(C_k \odot T(A))\Arrow{From Lemma \ref{unknown}}\\
    &= T^{\dagger}T(C_{P^{-1}(k)} \odot A)\\
    &= C_{P^{-1}(k)} \odot A\Arrow{From Proposition \ref{MP_equivalence_defn}}\\ 
    &= C_{P^{-1}(k)} \odot T^{\dagger}(X)\\
    &= C_{P^{-1}(k)} \odot T^{\dagger}(B).
\end{DispWithArrows*}
It follows that $T^{\dagger}$ preserves Hadamard circulant majorization, by using Theorem \ref{Had_circ_maj_charcaterization)}. 
\end{proof}

\begin{rem}
    It is easy to see that, if $T_1,T_2\colon M_n\rightarrow M_n$ preserve Hadamard circulant majorization, then so does $T_1\circ T_2$. For a self adjoint operator $T\colon M_n\rightarrow M_n$, we know that the group inverse always exists and $T\hash=T^\dagger$. This, along with the formula $T^\dagger=T^*(TT^*)^\dagger$, yields an alternate proof of Theorem \ref{MP}, suggested by the referee, as follows.
    Let $T\colon M_n\rightarrow M_n$ preserve Hadamard circulant majorization. From Theorem \ref{adj}, $T^*$ preserves and hence $TT^*$ preserves. Since $TT^*$ is self adjoint, $(TT^*)^\dagger=(TT^*)\hash$. Hence, from Corollary \ref{gr}, $(TT^*)^\dagger$ preserves. Therefore, $T^\dagger=T^*(TT^*)^\dagger$ preserves Hadamard circulant majorization.\\
\end{rem}

Here is an illustration of Theorem \ref{MP}.\\

\begin{ex}\label{Hadmarad_maj_counterexample}
Let $T\colon M_3 \rightarrow M_3$ be defined by $T(X) := x_{11} C_1+ x_{12} C_2 + x_{13} C_3$ for $X = \left[ x_{ij} \right] \in M_3$. Then, for $Y=[y_{ij}]\in M_3$, 
\[ T^{\dagger}(Y) = \frac{1}{3} \left( (y_{12} + y_{23} + y_{31}) E_{11} + (y_{13} + y_{21} + y_{32}) E_{12} + (y_{11} + y_{22} + y_{33}) E_{13} \right). \]
One may verify that the following hold: 
$T(C_1 \odot X) = C_2 \odot T(X), T(C_2 \odot X) = C_3 \odot T(X), T(C_3 \odot X) = C_1 \odot T(X), T^{\dagger}(C_1 \odot X) = C_3 \odot T^{\dagger}(X), T^{\dagger}(C_2 \odot X) = C_1 \odot T^{\dagger}(X)$ and $T^{\dagger}(C_3 \odot X) = C_2 \odot T^{\dagger}(X)$. By Theorem  \ref{Had_circ_maj_charcaterization)}, we may conclude that both $T$ and $T^{\dagger}$ preserve Hadamard circulant majorization.

\end{ex}

\vspace{.7cm}

{\bf Acknowledgements}:\\
The authors thank the referees for suggestions that have led to an improved presentation. P.R. Raickwade acknowledges funding received from the Prime Minister’s Research Fellowship (PMRF), Ministry of Education, Government of India, for carrying out this work.\\

{\bf Declarations}:\\
{\bf Conflict of interest}: The authors declare that they have no conflict of interest.

\end{document}